\documentclass[letterpaper, 10 pt, conference]{ieeeconf}  

\IEEEoverridecommandlockouts                              

\usepackage[hidelinks=true,colorlinks=true,citecolor=blue,linkcolor=blue,urlcolor=blue]{hyperref}       
\usepackage{url}            
\usepackage{booktabs}       
\usepackage{amsfonts}       
\usepackage{nicefrac}       
\usepackage{microtype}      
\usepackage{xcolor}         

\usepackage{algorithm}
\usepackage{algorithmic}

\usepackage{amsmath,amssymb}
\usepackage{bm}
\usepackage{graphicx}
\usepackage{ascmac}
\usepackage{here}
\usepackage{cite}
\usepackage{color}
\usepackage{graphicx}
\usepackage{mathrsfs}
\usepackage{subcaption}
\usepackage{wrapfig}
\usepackage{tikz}
\usetikzlibrary{positioning}

\usepackage{tabularx}
\usepackage{booktabs}

\newcommand{\calB}{{\mathcal B}}    
\newcommand{\calC}{{\mathcal C}}    
    
\newcommand{\calE}{{\mathcal E}}    
\newcommand{\calF}{{\mathcal F}}

\newcommand{\calN}{{\mathcal N}}    
    
\newcommand{\calP}{{\mathcal P}}

\newcommand{\bbE}{{\mathbb E}}

\newcommand{\bbP}{{\mathbb P}}    
    
\newcommand{\bbR}{{\mathbb R}}

\newcommand{\bbU}{{\mathbb U}}

\newcommand{\rmd}{{\rm d}}

\newcommand{\dimx}{n_x}
\newcommand{\dimu}{n_u}

\DeclareMathOperator*{\argmax}{arg\,max}
\DeclareMathOperator*{\argmin}{arg\,min}

\newcommand{\diag}{\mathop{\rm diag}\nolimits}

\newcommand{\trace}{{\rm tr}}

\newcommand{\fin}[1]{{\color{black}{#1}}}

\newcommand{\what}[1]{\widehat{#1}}

\newtheorem{theorem}{\bf Theorem}
\newtheorem{proposition}[theorem]{\bf Proposition}
\newtheorem{lemma}[theorem]{\bf Lemma}
\newtheorem{corollary}[theorem]{\bf Corollary}

\newtheorem{definition}[theorem]{\bf Definition}
\newtheorem{assumption}[theorem]{\bf Assumption}

\title{\LARGE \bf
Optimal Centered Active Excitation in Linear System Identification
}

\author{Kaito Ito and Alexandre Proutiere
\thanks{This work was supported in part by JSPS KAKENHI under Grant Number JP24K17297; in part by JST, ASPIRE under Grant Number JPMJAP2402; and in part by the Wallenberg AI, Autonomous Systems and Software Program (WASP) funded
by the Knut and Alice Wallenberg Foundation, the Swedish
Research Council (VR), and Digital Futures.}
\thanks{Kaito Ito is with the Department of Information Physics and Computing, The University of Tokyo, Tokyo 113-8654, Japan {\tt \small kaito@g.ecc.u-tokyo.ac.jp}}%
\thanks{Alexandre Proutiere is with the Division of Decision and Control Systems, School of Electrical Engineering and Computer Science, KTH Royal Institute of Technology, Stockholm 114~28, Sweden {\tt\small alepro@kth.se}}%
}

\begin{document}

\maketitle
\thispagestyle{empty}
\pagestyle{empty}

\begin{abstract}

We propose an active learning algorithm for linear system identification with optimal centered noise excitation. Notably, our algorithm, based on ordinary least squares and semidefinite programming, attains the minimal sample complexity while allowing for efficient computation of an estimate of a system matrix. More specifically, we first establish lower bounds of the sample complexity for any active learning algorithm to attain the prescribed accuracy and confidence levels. Next, we derive a sample complexity upper bound of the proposed algorithm, which matches the lower bound for any algorithm up to universal factors. 
Our tight bounds are easy to interpret and explicitly show their dependence on the system parameters such as the state dimension.

\end{abstract}

\section{Introduction}

System identification refers to a fundamental class of problems focused on designing procedures to estimate the unknown parameters of dynamical systems. It has a rich history within the control and learning communities. Most research has concentrated on linear systems, a crucial model class that arises in diverse fields such as finance, biology, robotics, and various engineering applications \cite{Ljung:1986:SIT:21413}. The state dynamics for such a system are: for a given initial state $x_0\in \bbR^{n_x}$, 
\begin{equation}\label{eq:lti} 
\forall t\ge 0, \quad x_{t+1} = Ax_t + Bu_t + w_t ,
\end{equation}
where the system matrix $A \in \mathbb{R}^{n_x \times n_x}$ is unknown and to be estimated, $u_t$ denotes the known input or excitation at time $t$, $B$ is the known control matrix, and $w_t$ represents the process noise. The central goal of system identification is to estimate $A$ from observations, typically by analyzing a trajectory of the system's successive states. 

Early foundational works focused on asymptotic analysis under passive excitation, where the input sequence $\{u_t\}_{t \ge 0}$ is fixed in advance and not adaptively tuned; see for instance \cite{Goodwin:1977:DSI,Ljung:c1,Ljung:c2}. While there has been longstanding interest in active excitation, designing the input to optimize estimation \cite{Mehra1976, Goodwin:1977:DSI, Bombois2011}, these efforts lack  finite-time performance guarantees. More recently, advances in random matrix theory \cite{mendelson:2014, vershynin:2012} and self-normalized processes \cite{pena2008self} have enabled a deeper understanding of the convergence rates and sample complexities of classical estimators, particularly the Least-Squares Estimator (LSE), for both observable \cite{rantzer2018, faradonbeh:2018:c2, simchowitz2018learning, oymak2019non, sarkar2019near, Jedra2020, sarkar2021finite, Jedra2023} and unobservable systems \cite{oymak2019non, sarkar2021finite}. However, these results are largely limited to passive excitation. 

Active excitation with finite-time statistical guarantees has been studied in \cite{Wagenmaker2020}, which presents a lower bound on the sample complexity for periodic excitation and proposes an algorithm that approaches this bound. Nevertheless, this lower bound fails to capture the correct dependence on the state dimension, and the algorithm relies on solving a non-convex program, raising concerns about feasibility and computational tractability. 
Such a computational issue persists in \cite{Chatzikiriakos2025high}, which focuses on identification from a finite set of system matrices. While \cite{Chatzikiriakos2025hidden} introduces a convex reformulation of active excitation for ARX systems and establishes a sample complexity upper bound, the optimality of the proposed method remains unaddressed due to the absence of corresponding sample complexity lower bounds.

In this work, we address these limitations under a natural and practically relevant assumption: the excitation process $\{u_t\}$ is a centered sub-Gaussian process, as is common in control applications due to their simplicity and ease of implementation~\cite{Ljung:1986:SIT:21413}. 
The main contributions are summarized as follows.
We derive sample complexity lower bounds satisfied by any algorithm using active excitation and achieving prescribed accuracy and confidence levels, parameterized by $(\varepsilon,\delta) $. Notably, the lower bounds in Theorem~\ref{thm:dim_lower_bound} and Corollary~\ref{Cor:lower_bound_infty_dim} exhibit a tight dependence on the state dimension. These bounds suggest that excitation inputs which maximize the minimum eigenvalue of the finite-time reachability (or controllability) Gramian can lead to algorithms with minimal sample complexity. 
Motivated by this, we propose an active learning algorithm based on the LSE. Since the system matrix $A$ is initially unknown and cannot be directly used to optimize the input covariance for active excitation, our method adaptively updates the input covariance using an estimate of $A$ obtained after a carefully designed initial phase. The resulting optimization problem can then be efficiently solved using semidefinite programming (SDP) solvers.
We establish a sample complexity upper bound for the proposed algorithm (Theorem~\ref{thm:upper_bound}) under the assumption that resetting the state is possible. When $\varepsilon$ and $\delta$ are small, this upper bound matches the tight lower bound in Corollary~\ref{Cor:lower_bound_infty_dim}, up to a multiplicative factor independent of $(\varepsilon, \delta)$ and the system matrix $A$.

\paragraph*{Related work}

While noise input has been commonly used for system excitation in the system identification literature~\cite{Ljung:1986:SIT:21413}, theoretical understanding of its optimal design with finite-time performance guarantees remains limited. To our knowledge, this work is the first to establish a sample complexity guarantee for optimal noise excitation. Previous works such as \cite{sarkar2019near, simchowitz2018learning} derived sample complexity upper bounds for the LSE under passive excitation and mentioned that these bounds could be extended to settings involving noise input. However, these studies do not address the optimization of the noise input for efficient system excitation. 
Furthermore, the dependence of the bounds in \cite{sarkar2019near} on the state dimension is implicit, with the dependence on $A$ hidden in constants, while the bounds in \cite{simchowitz2018learning} are difficult to interpret. The recent work \cite{Jedra2023} provides a more interpretable sample complexity upper bound for the LSE under passive excitation only, explicitly characterizing its dependence on both $A$ and the state dimension. Our sample complexity analysis is inspired by their approach.

\paragraph*{Notation}
The identity matrix of appropriate dimension is denoted by $ I $. For a symmetric matrix $ S $, $ S \succeq 0 $ means that $ S $ is positive semidefinite. The spectral norm and the Frobenius norm for matrices are denoted by $ \| \cdot \| $ and $ {\| \cdot \|_{\rm F}} $, respectively. For a vector, $ \| \cdot \| $ denotes the Euclidean norm. The spectral radius of a square matrix $ A $ is denoted by $ \rho(A) $.
The minimum eigenvalue of $ A $ is denoted by $ \lambda_{\min} (A) $.
The expectation is denoted by $ \bbE $. The multivariate Gaussian distribution with mean $ \mu $ and covariance $ \Sigma $ is denoted by $ \calN(\mu,\Sigma) $.

\section{Sample Complexity Lower Bounds under Active Excitation}

In this section, we first present the system identification problem setting and then derive lower bounds on the sample complexity required to estimate $A$ with specified accuracy and confidence levels.

\subsection{Problem Setting}
We consider a discrete-time linear system \eqref{eq:lti}, where $ x_t \in \bbR^{\dimx} $ and $ u_t \in \bbR^{\dimu} $. The process noise $ \{w_t\}_{t\ge 0} $ is i.i.d., with $ \bbE[w_t] = 0 $ and $ \bbE[w_t w_t^\top] = I $.
Without loss of generality, we assume that the initial state is given by $ x_0 = 0 $.
We wish to estimate $ A \in \bbR^{\dimx \times \dimx}$ based on the observation of a sequence $ \{x_0,\ldots,x_t,u_0,\ldots,u_{t-1}\} $. In this section, we assume that this sequence is obtained by a single trajectory without state resets. Under this setting, we consider algorithms with adaptive excitation. Such an algorithm consists in (i) a control policy (or a sampling rule) defining, in each step $t$, the conditional distribution of $u_t$ given $ \calF_t $, where $ \calF_t $ is the $ \sigma $-algebra generated by $ \{x_0,\ldots,x_t,u_0,\ldots,u_{t-1}\} $; (ii) 
a $ \calF_t $-measurable estimate $ \what{A}_t $ of $ A $. To derive problem-specific bounds of the sample complexity, we need to consider a class of algorithms that adapt to the matrix $A$. Indeed, an algorithm always returning $A$ would not need any sample if the true matrix is $A$, but would fail for any other matrix. We consider the class of locally stable algorithms as defined below. 
\begin{definition}[$ (\varepsilon,\delta) $-locally stable algorithm~\cite{Jedra2023}]
    For $ \varepsilon > 0 $ and $ \delta \in (0,1) $, an algorithm returning an estimate $ \what{A}_t $ of $ A $ at time $ t $, is said to be $ (\varepsilon,\delta) $-locally stable in $ A $ if there exists a finite deterministic time $ \tau $ such that for any $ t \ge \tau $,
    \begin{equation}\label{def:local_stability}
        \inf_{A' \in \calB(A,6\sqrt{2}\varepsilon)} \bbP_{A'} \! \left( \|\what{A}_t - A' \| \le \varepsilon \right) \ge 1- \delta,
    \end{equation}
    where $ \calB (A,6\sqrt{2}\varepsilon) := \left\{ A'\in \bbR^{\dimx \times \dimx} : \| A - A' \| \le 6\sqrt{2} \varepsilon \right\} $.
    \hfill $ \diamondsuit $
\end{definition}

The choice of the constant $ 6\sqrt{2} $ is made solely for technical convenience in the subsequent analysis and does not play any essential role.
The sample complexity $ \tau_{A} $ of an $ (\varepsilon,\delta) $-locally stable algorithm is defined as the minimal number $ \tau $ such that \eqref{def:local_stability} holds for any $ t \ge \tau $.

In this paper, we restrict our attention to the class of algorithms with centered excitation. The control policy of an algorithm is defined by a sequence $ \{\pi_t \} $ of conditional densities such that for any Borel set $ \bbU \subseteq \bbR^{\dimu} $ and almost surely,
\begin{align}
      \bbP(u_t \in \bbU \mid \calF_{t}) = \int_{\bbU} \pi_t (u \mid x_1,\ldots,x_t,u_0,\ldots,u_{t-1}) \rmd u . \label{eq:density}
\end{align}
We say that the excitation is centered if the following assumption holds. 

\begin{assumption}\label{ass:independent_input}
    The control policy $ \{\pi_t\} $ satisfies $ \bbE[u_t | {\cal F}_{t}] = 0 $ and $ \bbE[ \|u_t\|^2 ] < \infty $ for any $ t\ge 0 $. 
    \hfill $ \diamondsuit $
\end{assumption}

Under the constraint $ \bbE[u_t | {\cal F}_{t}] = 0 $, the excitation input $ u_t $ can be adaptively tuned using $ \{x_0,\ldots,x_t,u_0,\ldots,u_{t-1}\} $, enabling more efficient excitation in a non-i.i.d. manner.
In the subsequent sections, we show that this class of active excitation leads to a computationally efficient optimal algorithm, in contrast to periodic excitation~\cite{Wagenmaker2020}.

\subsection{Sample Complexity Lower Bounds}\label{subsec:lower_bound}
We derive lower bounds of the sample complexity for {\em any} $ (\varepsilon,\delta) $-locally stable algorithms.
First, we provide a sample complexity lower bound that holds for any $ A $ and $ \varepsilon $.
For the derivation, we utilize a change-of-measure argument and the data processing inequality as in \cite{Jedra2023}, and the proof is given in Appendix~\ref{app:lower_bound}.

\begin{proposition}[Lower bound for any $ A $ and $ \varepsilon $]\label{prop:lower_bound_persis}
    Suppose that $ w_t \sim \calN(0,I) $ for any $ t \ge 0 $. 
   Then, for any matrix $ A $, for all $ \varepsilon > 0 $ and $ \delta \in (0,1) $, the sample complexity $ \tau_A $ of any $ (\varepsilon,\delta) $-locally stable algorithm in $ A $ with centered excitation satisfies
   \begin{equation}\label{eq:lower_bound_persis}
        \lambda_{\min} \! \left( \sum_{s=1}^{\tau_A-1} \Sigma_s   \right) \ge \frac{1}{2\varepsilon^2} \log \! \left(\frac{1}{3 \delta}\right) ,
   \end{equation}
   where $ \Sigma_s := \bbE[x_s x_s^\top] $.
   Moreover, $ \Sigma_s $ satisfies for any $ s \ge 0 $,
   \begin{align}
    \Sigma_{s+1} = A \Sigma_s A^\top + B U_s B^\top + I ,  \label{eq:evolution_sigma}
    \end{align}
    where $ U_s := \bbE[u_s u_s^\top] $. 
   \hfill $ \diamondsuit $
\end{proposition}

Observe that the lower bound \eqref{eq:lower_bound_persis} depends on the excitation inputs $ \{u_s\} $ only through their covariances $ \{U_s\} $. This implies that for designing centered excitation inputs, it suffices to consider their covariances $ \{U_s\} $ instead of control policies.
Moreover, by using \eqref{eq:evolution_sigma}, we have
\begin{equation}
    \sum_{s=1}^{t -1} \Sigma_s = \sum_{s=1}^{t -1} \! \left( \Gamma_s(A) + \Xi_s (A, \{U_k\}_{k=0}^{s-1}) \right) , \nonumber
\end{equation}
where $ \Gamma_s(A) := \sum_{k=0}^{s-1} A^k (A^k)^\top $ and $ \Xi_s (A,\{U_k\}_{k=0}^{s-1}) := \sum_{k=0}^{s-1} A^{s-1-k} BU_k B^\top (A^{s-1-k})^\top $ are reachability Gramians. 
When $ U_k \equiv U $, we write $ \Xi_s(A,U) = \Xi_s (A,\{U_k\}_{k=0}^{s-1}) $.
Setting $ B = 0 $ recovers the sample complexity lower bound without active excitation~\cite[Theorem~1]{Jedra2023}.
It should be noted that the above result does not assume the Gaussianity of $ u_t $. 

If $ A $ is stable and $ \varepsilon > 0 $ is small, that is, $ \rho(A) <1 $ and high accuracy of the estimation is required, we can derive a tighter lower bound of the sample complexity, which shows its explicit dependence on the state dimension $ \dimx $.
To this end, we introduce the following power constraint on control policies.
\begin{assumption}[Power constraint]\label{ass:power}
   For system~\eqref{eq:lti} and $ \bar{u} \ge 0 $, the control policy $ \{\pi_t\} $ satisfies $ \bbE[ \|u_t \|^2] \le \bar{u} $ for any $ t \ge 0 $.
   \hfill $ \diamondsuit $
\end{assumption}
\fin{For a set $ \Theta $ of system matrices, we say that Assumptions~\ref{ass:independent_input} and \ref{ass:power} hold uniformly over $ \Theta $ if, for any $A' \in \Theta$ and $t \ge 0$, it holds that $\mathbb{E}_{A'}[u_t | \mathcal{F}_t] = 0 $ and
$\mathbb{E}_{A'}[\|u_t\|^2] \le \bar u$.}

Define the infinite-time reachability Gramian $ \Gamma_\infty(A) := \sum_{s=0}^\infty A^s (A^s)^\top $ (it is indeed well-defined for stable $ A $). The tighter bound is obtained by constructing a suitable packing of a set of system matrices, and the proof is deferred to Appendix~\ref{app:dim_lower_bound}.
\begin{theorem}[Lower bound for stable $ A $ and small $ \varepsilon $]\label{thm:dim_lower_bound}
   Suppose that $ w_t \sim \calN(0,I) $ for any $ t \ge 0 $.
    Then, there exists a universal constant $ c > 0 $ such that for any stable matrix $ A $, for all $ \varepsilon \in (0, \|\Gamma_\infty (A) \|^{-3}/(12\sqrt{2})) $ and $ \delta \in (0,1/2) $, the sample complexity $ \tau_A $ of any $ (\varepsilon,\delta) $-locally stable algorithm in $ A $ with centered power-constrained excitation (verifying Assumptions \ref{ass:independent_input} and \ref{ass:power} \fin{uniformly over $ \calB(A,6\sqrt{2}\varepsilon) $}) satisfies
    \begin{align}
        \lambda_{\min} \! \left( \sum_{s=1}^{\tau_A-1} \Sigma_s   \right) \ge \frac{c}{ (1+ \| B\|^2 \bar{u}) \varepsilon^2} \! \left( \log \! \left( \frac{1}{\delta}  \right) + \dimx \right) . \label{eq:lower_bound_dim}
    \end{align}
   \hfill $ \diamondsuit $
\end{theorem}

The lower bounds of the sample complexity~\eqref{eq:lower_bound_persis},~\eqref{eq:lower_bound_dim} motivate us to maximize $ \lambda_{\min} (\sum_{s=1}^{\tau-1} \Sigma_s) $ with respect to $ \{U_s\} $ under a power constraint $ \trace(U_s) \le \bar{u} $.
\begin{align}
	&\underset{\{U_s\}_{s=0}^{\tau-2} }{\textrm{maximize}}  ~~ && \lambda_{\min} \! \left(\sum_{s=1}^{\tau-1} \Sigma_s\right) , \label{prob:max_Aid_persis}\\
	&\textrm{subject to} \ && \Sigma_{s+1} = A\Sigma_s A^\top + BU_s B^\top + I , \\
   & &&\qquad\qquad \forall s \in \{0,\ldots,\tau-2\}, \nonumber\\
    & && \Sigma_0 = 0, \nonumber\\
    & && \trace(U_s) \le  \bar{u}, ~~ U_s \succeq 0, ~~ \forall s \in \{0,\ldots,\tau-2\}. \nonumber
\end{align}

The above optimization problem can be reformulated as an SDP by replacing the objective by $ \lambda \in \bbR $ and adding a constraint $ \sum_{s=1}^{\tau-1} \Sigma_{s} \succeq \lambda I $.
However, for system identification, $ A $ is unknown, and even if $ A $ is known, this problem is computationally expensive to solve for large $ \tau $, which corresponds to small $ \varepsilon $ and $ \delta $.
To remedy the computational issue, we provide more tractable lower bounds, which are weaker than \eqref{eq:lower_bound_persis},~\eqref{eq:lower_bound_dim}, but become tighter as $ \tau $ becomes large.
The proof is shown in Appendix~\ref{app:lower_bound_infty}.
\begin{corollary}\label{Cor:lower_bound_infty}
    Suppose that $ w_t \sim \calN(0,I) $ for any $ t \ge 0 $.
   Then, for any stable matrix $ A $, for all $ \varepsilon > 0 $ and $ \delta \in (0,1) $, the sample complexity $ \tau_{A} $ of any $ (\varepsilon,\delta) $-locally stable algorithm in $ A $ with \fin{centered power-constrained excitation (verifying Assumptions~\ref{ass:independent_input} and \ref{ass:power})} satisfies
   \begin{align}
        \tau_{A} - 1 &\ge \frac{1}{2\varepsilon^2 {\displaystyle \max_{U\succeq 0, \trace(U) \le \bar{u}}} \lambda_{\min} \! \left(\Gamma_\infty (A) + \Xi_\infty (A,U) \right)} \nonumber\\
        &\quad \times\log \! \left(\frac{1}{3 \delta}\right) ,\label{eq:lower_bound_infty_persis}
   \end{align}
   where $ \Xi_\infty(A,U) := \sum_{s=0}^\infty A^s BUB^\top (A^s)^\top $. 
   \hfill $ \diamondsuit $
\end{corollary}
\begin{corollary}\label{Cor:lower_bound_infty_dim}
    Suppose that $ w_t \sim \calN(0,I) $ for any $ t \ge 0 $. Then, there exists a universal constant $ c > 0 $ such that for any stable matrix $ A $, for all $ \varepsilon \in (0, \|\Gamma_\infty (A) \|^{-3}/(12\sqrt{2})) $ and $ \delta \in (0,1/2) $, the sample complexity $ \tau_A $ of any $ (\varepsilon,\delta) $-locally stable algorithm in $ A $ with centered power-constrained excitation (verifying Assumptions \ref{ass:independent_input} and \ref{ass:power} \fin{uniformly over $ \calB(A,6\sqrt{2}\varepsilon) $}) satisfies
   \begin{align}
       &\tau_A - 1 \nonumber\\
       &\ge \frac{c}{ (1+ \| B\|^2 \bar{u}) \varepsilon^2 {\displaystyle \max_{U\succeq 0, \trace(U) \le \bar{u}}} \lambda_{\min} \! \left(\Gamma_\infty (A) + \Xi_\infty (A,U) \right)} \nonumber\\
       &\quad \times \left( \log \! \left( \frac{1}{\delta}  \right) + \dimx \right) . \label{eq:lower_bound_dimension}
   \end{align}
   \hfill $ \diamondsuit $
\end{corollary}

The optimization problem $ \max_{U\succeq 0, \trace(U) \le \bar{u}} \lambda_{\min} (\Gamma_\infty (A) + \Xi_\infty (A,U)) $ can be solved efficiently. In fact, it is known that if $ A $ is stable, $ \Gamma_\infty(A) + \Xi_\infty (A,U) $, which is the reachability Gramian of the system $ x_{t+1} = Ax_t + (I +BUB^\top)^{1/2} u_t $, is a unique solution to the Lyapunov equation $ P = APA^\top + I + BUB^\top $~\cite{Chen1984}. Moreover, since the maximization of $ \lambda_{\min} (P) $ is equivalent to maximizing $ \lambda $ subject to the constraint $ P \succeq \lambda I $, this problem can be rewritten as
\begin{align}
	&\underset{U\succeq 0, P \succeq 0, \lambda \ge 0 }{\textrm{maximize}}   && \lambda , \label{prob:max_Aid_SDP_infty}\\
	&\textrm{~~subject to}  && P - \lambda I \succeq  0 , \nonumber\\
    & && P - APA^\top - I - BUB^\top = 0 , \nonumber\\
    & && \trace(U) \le \bar{u} , \nonumber
\end{align}
which is an SDP, and thus, can be solved by convex optimization techniques. Notably, unlike problem~\eqref{prob:max_Aid_persis}, the above problem does not depend on $ \tau $, which ensures computational efficiency even for large $ \tau $.

\section{Finite-time Analysis of the LSE with Active Excitation}

Following the literature on system identification, we adopt the LSE to estimate $ A $.
The LSE for a given trajectory $ \{x_0,\ldots,x_t,u_0,\ldots,u_{t-1}\} $ up to time $ t $ is given by
\begin{equation}
    \forall t \ge 1, ~~ \what{A}_t \in \argmin_{A' \in \bbR^{\dimx \times \dimx}} \sum_{s=0}^{t-1} \| x_{s+1} - A' x_s - Bu_s \|^2  . \label{eq:LSE_persis}
\end{equation}
It admits the following closed-form solution:
\begin{equation}
    \forall t \ge 1, ~~ \what{A}_t = \left( \sum_{s=1}^{t-1} (x_{s+1} - Bu_s) x_s^\top  \right) \! \left( \sum_{s=1}^{t-1} x_s x_s^\top \right)^\dagger , \label{eq:LES_persis_explicit}
\end{equation}
where $ \dagger $ represents the Moore--Penrose inverse.
The estimation error is given by
\begin{align}
    \forall t \ge 1, ~~ \what{A}_t - A &= \left( \sum_{s=1}^{t-1} w_s x_s^\top  \right) \! \left( \sum_{s=1}^{t-1} x_s x_s^\top \right)^\dagger \nonumber\\
    &= W^\top X \left(X^\top X \right)^\dagger , \label{eq:error_closed_persis}
\end{align}
where $ X := [x_1,x_2,\ldots,x_{t-1}]^\top $ and $ W := [w_1,w_2,\ldots,w_{t-1}]^\top $.

\begin{algorithm}[tb]
    \caption{Active Learning of $ A $}
    \label{alg:learning}
    \begin{algorithmic}[1]
    \STATE \textbf{Input:} Input power $\bar{u}$, initial horizon $t_0$, projection $ \Pi : \bbR^{\dimx \times \dimx} \rightarrow \Theta_d $
    \FOR{$t = 0$ to $t_0 -1$}
        \STATE Inject $ u_t $ such that $ \bbE[u_t] = 0 $, $ \bbE[u_t u_t^\top] = (\bar{u}/\dimu) I $ to the system $ x_{t+1} = Ax_t + Bu_t + w_t $
        \STATE Obtain LSE $ \what{A}_{t+1} $ by \eqref{eq:LES_persis_explicit}
    \ENDFOR
    \IF{$ \rho(\what{A}_{t_0}) < 1 $}
        \STATE $ \bar{A}_{t_0} \leftarrow \what{A}_{t_0} $
    \ELSIF{$ \rho(\what{A}_{t_0}) \ge 1 $}
        \STATE Perform projection $ \bar{A}_{t_0} \leftarrow \Pi(\what{A}_{t_0}) $
    \ENDIF
    \STATE Solve $ \what{U} \in {\displaystyle \argmax_{U\succeq 0, \trace(U) \le \bar{u}}} \lambda_{\min} \! \left( \Gamma_\infty (\bar{A}_{t_0}) + \Xi_\infty (\bar{A}_{t_0}, U) \right) $
    \STATE Reset $ x_{t_0} \leftarrow 0 $
    \FOR{$t = t_0, t_0 + 1,\ldots $}
        \STATE Inject $ u_t $ such that $ \bbE[u_t | \calF_t] = 0 $, $ \bbE[u_t u_t^\top | \calF_{t_0}] = \what{U} $ to the system $ x_{t+1} = Ax_t + Bu_t + w_t $
        \STATE Obtain LSE $ \what{A}_{t+1} $ by \eqref{eq:LES_persis_explicit} using \fin{$ \{x_s \}_{s=t_0}^{t+1} $ and $ \{u_s\}_{s=t_0}^t $}
    \ENDFOR
    \end{algorithmic}
\end{algorithm}

As observed in Subsection~\ref{subsec:lower_bound}, the excitation input whose covariance maximizes $ \lambda_{\min} (\Gamma_\infty(A) + \Xi_\infty (A,U)) $ will effectively reduce the sample complexity. However, since we do not know $ A $ initially, we first use a predetermined input covariance to obtain an initial estimate of $ A $, and after that, we maximize the surrogate objective $ \lambda_{\min} (\Gamma_\infty(A) + \Xi_\infty (A,U)) $ where $ A $ is replaced by the initial estimate, and excite the system using the resulting input covariance. More formally, our estimation algorithm proceeds as follows:
\paragraph*{Estimation algorithm}
For the estimation of stable $ A $, we first estimate $ A $ using $ \{x_s,u_s\}_{s=0}^{t_0} $ with predetermined input covariance $ U_s = (\bar{u}/\dimu) I $ for $ s \in \{0,\ldots,t_0-1\}  $ satisfying $ \trace(U_s) \le \bar{u} $.
Since the LSE $ \what{A}_{t_0} $ can be unstable, we use a projection $ \Pi : \bbR^{\dimx \times \dimx} \rightarrow \Theta_d := \{ A \in \bbR^{\dimx \times \dimx} : \rho (A) \le 1 - d\} $, $ d > 0 $, which ensures that $ \bar{A}_{t_0} := \Pi(\what{A}_{t_0}) $ is stable and $ \| \what{A}_{t_0} - \bar{A}_{t_0} \| $ is small as possible.\footnote{Efficient algorithms that give a nearby stable approximation to a given unstable matrix are proposed, for example, in \cite{Orbandexivry2013,Gillis2019}.} If $ \what{A}_{t_0} $ is stable, then $ \bar{A}_{t_0} := \what{A}_{t_0} $.
Next, the input covariance $ U_s = \what{U} $ after time $ s = t_0 $ is designed so that it maximizes $ J_{\bar{A}_{t_0}} (\what{U}) := \lambda_{\min} \! \left( \Gamma_\infty (\bar{A}_{t_0}) + \Xi_\infty (\bar{A}_{t_0}, \what{U}) \right) $ under the constraint $ \trace(\what{U}) \le \bar{u} $. After resetting the state to $ x_{t_0} = 0 $, we use the designed input $ u_s $ satisfying $ \bbE[u_s u_s^\top | \calF_{t_0}] = \what{U} $ for any $ s \ge t_0 $ to excite the system and obtain the LSE $ \what{A}_{t} $ using the trajectory $ \{x_s,u_s\}_{s=t_0}^t $ only after $ s=t_0 $. The pseudocode is given in Algorithm~\ref{alg:learning}.
 \hfill $ \diamondsuit $

In the above algorithm, we assume that resetting the state is possible and do not leverage the data $ \{x_s,u_s\}_{s=0}^{t_0-1} $ used in the initial phase to estimate $ \what{A}_t $ for $t>t_0$. These assumptions are made solely to simplify the sample complexity analysis. In practice, retaining the trajectory data over the initial horizon without resetting the state \fin{may} improve the sample complexity. 
Nevertheless, we will see in Theorem~\ref{thm:upper_bound} that the above algorithm with suitable initial horizon $ t_0 $ achieves a nearly optimal sample complexity.

For $ \what{U} $ in Algorithm~\ref{alg:learning}, define $ Q_{\what{U}} := B\what{U}B^\top + I $ and $ \eta_t := ((\bar{u}/\dimu )B B^\top + I)^{-1/2} (Bu_t + w_t)  $ for $ t \in \{0,\ldots,t_0-1\} $ and $ \eta_t := Q_{\what{U}}^{-1/2} (Bu_t + w_t) $ for $ t \ge t_0 $.
For the finite-time analysis of the LSE, we assume the sub-Gaussianity of $ \eta_t $~\cite{Vershynin2018}.
\begin{assumption}[Sub-Gaussianity]\label{ass:subgauss}
    The $ \psi_2 $-norm of each component of $ \eta_t $ is bounded from above by $ K $,\footnote{\label{note3}The $ \psi_2 $-norm of an $ \bbR $-valued random variable $ X $ is defined as $ \| X \|_{\psi_2} := \inf \{ K > 0 : \bbE[\exp(X^2/K^2)] \le 2 \} $. If $ \| X \|_{\psi_2} < \infty  $, then $ X $ is said to be sub-Gaussian. If $ X\sim \calN(0,\sigma^2) $, then $ \| X \|_{\psi_2} = C_g \sigma $, where $ C_g $ is a universal constant.} and $ \eta_t $ has independent components for any $ t \ge 0 $.
    \hfill $ \diamondsuit $
\end{assumption}

If $ u_t $ and $ w_t $ have sub-Gaussian components, then $ \eta_t $ is also sub-Gaussian and has a finite $ \psi_2 $-norm. For example, if $ w_t \sim \calN(0,I) $ and $u_t \sim \calN(0,(\bar{u}/\dimu) I )$ for $t \le t_0-1$ and $ u_t | \what{U} \sim \calN(0,\what{U}) $ for $t \ge t_0$, then $ \eta_t = [\eta_{t,1},\ldots,\eta_{t,\dimx}]^\top \sim \calN(0,I) $ has independent components and $ \| \eta_{t,i} \|_{\psi_2} = C_g $ (where $C_g$ is a universal constant).

Now, we derive an upper bound of the sample complexity of Algorithm~\ref{alg:learning}.
Let $ \varepsilon_{t_0} > 0 $ be the minimum $ \varepsilon $ such that \eqref{eq:upper_bound_persis} in Subsection~\ref{subsec:upper_bound_proof} with $ t = t_0 $, $ U =  (\bar{u}/\dimu) I $, and $ \delta \leftarrow \delta/2 $ is satisfied. Define $ U^* (A) $ as any element of $ \argmax_{U\succeq 0, \trace(U) \le \bar{u}} \lambda_{\min} ( \Gamma_\infty (A) + \Xi_\infty (A,U) ) $ and $ \bar{\gamma} := (1 + \|B\|^2 \bar{u}) (\sum_{s=0}^\infty \| A^s \| )^2  < \infty $.

\begin{theorem}[Sample complexity upper bound]\label{thm:upper_bound}
    Suppose that Assumption \ref{ass:subgauss} holds.
    Assume that $ t_0 $ is chosen such that $ \varepsilon_{t_0} \le \| \Gamma_\infty (A) \|^{-3/2} /4 $. Assume also that $ \{u_t\}_{t=0}^{t_0 -1} $ are mutually independent, and $ \{u_t\}_{t\ge t_0} $ are mutually independent, conditioned on $ \{x_{s+1},u_s\}_{s=0}^{t_0-1} $.
    Then, for all $ \varepsilon > 0 $ and $ \delta \in (0,1) $, Algorithm~\ref{alg:learning} yields $ \bbP \! \left( \| \what{A}_t - A \| > \varepsilon  \right) \le \delta $ for $ t \ge t_0 $, as long as the following condition holds:
    \begin{align}
        t &\ge 1+ \frac{C}{\lambda_{\min} (\Gamma_\infty (A) + \Xi_\infty (A,U^*(A)))} \max \! \left\{ \frac{1}{\varepsilon^2} , \bar{\gamma} \right\} \nonumber\\
         &\quad \times\left( \log \! \left( \frac{2}{\delta}\right) + \dimx  \right) \nonumber\\
         &+ \frac{\|\Gamma_\infty(A) \|^2}{\lambda_{\min} (\Gamma_\infty (A) + \Xi_\infty (A,U^*(A)))} \Bigl(C_1  (t-1)\varepsilon_{t_0}  + C_2 t_0  \nonumber\\
         & + C_2 (t-1) (1-\|\Gamma_\infty(A) \|^{-1})^{t-1} + C_2 (\|\Gamma_\infty(A)\| - 1) \Bigr)  \label{eq:upper_bound_algo} ,
    \end{align}
    where $ C_1 := 64 (1 + \| B\|^2 \bar{u})  \| \Gamma_\infty(A) \| $, $ C_2 := 1+\|B\|^2 \bar{u}  $, $ C := c' K^4 $, and $ c' > 0 $ is a universal constant.
    Moreover, for any $ t \ge 0 $, it holds that $ \bbE[\|u_t \|^2] \le \bar{u} $.
    \hfill $ \diamondsuit $
\end{theorem}

The second term of the right-hand side of \eqref{eq:upper_bound_algo} coincides with the sample complexity lower bound~\eqref{eq:lower_bound_dimension} up to a multiplicative factor independent of $ (\varepsilon,\delta) $, $ \dimx $, and $ A $ in a regime where $ \varepsilon $ is small such that $\varepsilon^2 \le \bar{\gamma}^{-1} $.
Since $ 1-\|\Gamma_\infty(A) \|^{-1} \in (0,1) $, as $ t $ increases, $ C_2 (t-1) (1-\|\Gamma_\infty(A) \|^{-1})^{t-1} $ in the third term becomes rapidly negligible.
The terms $ C_1 (t-1) \varepsilon_{t_0} $ and $ C_2 t_0 $ represent the losses due to the discrepancy between $ U^* (A) $ and $ U^*(\what{A}_{t_0}) $ and not using $ \{x_s,u_s\}_{s=0}^{t_0 -1} $ for the estimate, respectively.

We will establish in Proposition~\ref{prop:LSE_performance_persis} in Subsection~\ref{subsec:upper_bound_proof} that $ \varepsilon_{t_0} $ scales as $ \varepsilon_{t_0} = O(1/\sqrt{t_0}) $. Ignoring the term $ C_2 (t-1) (1-\|\Gamma_\infty(A) \|^{-1})^{t-1} $, the third term in \eqref{eq:upper_bound_algo} is minimized by $ t_0 = \Theta( (t-1)^{2/3} ) $, and its minimum value is $ \Theta ((t-1)^{2/3}) $. 
Therefore, for small $\varepsilon$ and $\delta$, the linear term $ t $ dominates the sub-linear term $ \Theta((t-1)^{2/3}) $, and the sample complexity upper bound \eqref{eq:upper_bound_algo} coincides with the lower bound~\eqref{eq:lower_bound_dimension} up to multiplicative and additive factors.

\subsection{Proof Sketch of Theorem~\ref{thm:upper_bound}}\label{subsec:upper_bound_proof}
We provide a sketch of the proof of Theorem~\ref{thm:upper_bound}. The detailed proof is given in Appendix~\ref{app:upper_bound}.
The first key ingredient of the proof is a performance guarantee for the LSE when using a fixed excitation covariance $ U_s \equiv U $.
Similar to the case without control input~\cite[Theorem~1]{Jedra2020}, we obtain the following result. The proof is provided in Appendix~\ref{app:LSE_performance}.
\begin{proposition}[LSE performance]\label{prop:LSE_performance_persis}
   Fix any $ U \succeq 0 $, and let $ \{u_t\} $ be zero-mean, mutually independent random \fin{vectors} such that $ \bbE[u_t u_t^\top] =  U $ for any $ t \ge 0 $. Suppose that $ \eta_t := Q_{U}^{-1/2} (Bu_t + w_t) $ satisfies Assumption~\ref{ass:subgauss}.
    Let $ \what{A}_t $ be the LSE given in \eqref{eq:LES_persis_explicit}. Then, for all $ \varepsilon > 0 $ and $ \delta \in (0,1) $, we have $ \bbP \! \left( \| \what{A}_t - A \| > \varepsilon  \right) \le \delta $, as long as the following condition holds:
    \begin{align}
       & \lambda_{\min} \! \left(  \sum_{s=1}^{t-1} \left( \Gamma_s(A) + \Xi_s (A,U) \right) \right) \nonumber\\
       &\qquad \ge C \max \left\{ \frac{1}{\varepsilon^2} , \|\Gamma_{U} \|^2 \right\} \left( \log \left( \frac{1}{\delta}\right) + \dimx  \right),\label{eq:upper_bound_persis}
    \end{align}
    where $ C := c' K^4 $, $ c' > 0 $ is a universal constant, and $ \Gamma_{U} $ is defined in Appendix~\ref{app:LSE_performance}.
    \hfill $ \diamondsuit $
\end{proposition}

In the above proposition, the matrix $\Gamma_{U}$ satisfies $ \| \Gamma_{U} \|^2 \le  (\sum_{s=0}^{t-2} \| A^s Q_U^{1/2} \| )^2  \le \bar{\gamma} $ for any $ U \succeq 0 $ such that $ \trace(U) \le \bar{u} $.

Next, we need to relate the above condition~\eqref{eq:upper_bound_persis} to the objective $ J_{A} (U) := \lambda_{\min} \left( \Gamma_\infty (A) + \Xi_\infty (A,U) \right) $. The most crucial part in the analysis is bounding the difference between the maximum value $ J_{A} (U^*(A)) $ and $ J_{A} (U^*(\what{A}_{t_0})) $ attained by a maximizer of the surrogate objective $ J_{\what{A}_{t_0}} $.
The following result is the second key ingredient of the proof of Theorem \ref{thm:upper_bound}. It provides the sensitivity of the optimal solution $ U^* $ under a system perturbation bound $ \| A -\what{A} \| \le \varepsilon $. The proof is given in Appendix~\ref{app:perturbation_modi}. There, the sensitivity of the Gramians $ \Gamma_\infty $ and $ \Xi_\infty $ to perturbations plays an important role  (refer to Lemma~\ref{lem:perturbed_gramian} in Appendix \ref{app:dim_lower_bound}).
\begin{proposition}\label{prop:perturbation_infty_persis_modi}
    Let $ A $ be a stable matrix. Then, for any $ \Delta \in \bbR^{\dimx \times \dimx} $ and $ U \succeq 0 $ such that $ \| \Delta \| \le \|\Gamma_\infty (A) \|^{-3/2}/4 $ and $ \trace(U) \le \bar{u} $, it holds that
    \begin{align}
        &J_{A} (U^* (A)) -  J_{A} (U^* (A+\Delta)) \nonumber\\
        &\qquad \le 64 (1 + \| B\|^2 \bar{u}) \| \Delta \| \| \Gamma_\infty(A) \|^3 .
    \end{align}
\hfill $ \diamondsuit $
\end{proposition}

Now, we are ready to sketch the proof of the upper bound of the sample complexity of our algorithm. \\
\paragraph*{Proof sketch of Theorem~\ref{thm:upper_bound}}
    Let $ U_{t_0} := U^* (\bar{A}_{t_0}) $.
    The assumption $ \varepsilon_{t_0} \le \| \Gamma_\infty(A) \|^{-3/2} /4 $ ensures that $ \lambda_{\min} (\Gamma_\infty (A) + \Xi_\infty (A,U_{t_0})) $ is close enough to $ \lambda_{\min} (\Gamma_\infty (A) + \Xi_\infty (A,U^* (A))) $ with probability at least $ 1-\delta/2 $ by Proposition~\ref{prop:perturbation_infty_persis_modi}, which results in the term $ C_1  (t-1)\varepsilon_{t_0} $ in the upper bound~\eqref{eq:upper_bound_algo}.
    Moreover, by utilizing the stability of $ A $, the difference between $ \lambda_{\min} (\Gamma_\infty (A) + \Xi_\infty (A,U_{t_0})) $ and $ \lambda_{\min} \! \left( \sum_{s=1}^{t-t_0- 1} ( \Gamma_s(A) + \Xi_s (A,U_{t_0}) ) \right) $ can be bounded.
    Consequently, we arrive at
    \begin{align}
    &(t-1) \lambda_{\min} (\Gamma_\infty (A) + \Xi_\infty (A,U^* (A))) \nonumber\\
    &\quad - \lambda_{\min} \! \left( \sum_{s=1}^{t-t_0 - 1} (  \Gamma_s(A) + \Xi_s(A,U_{t_0}) ) \right) \nonumber\\
    &\le \bigl( C_1  (t-1)\varepsilon_{t_0} +  C_2  (t-1) (1-\|\Gamma_\infty(A) \|^{-1})^{t-1} \nonumber\\
    &\quad + C_2 (\|\Gamma_\infty(A) \| - 1)  + C_2 t_0  \bigr) \| \Gamma_\infty(A) \|^2 = : \bar{C} . \nonumber
    \end{align}

Therefore, if $ \| \bar{A}_{t_0} - A \| \le \varepsilon_{t_0} $ and
\begin{align}
   &(t-1) \lambda_{\min} (\Gamma_\infty (A) + \Xi_\infty (A,U^* (A))) \nonumber\\
   &\ge C \max \left\{ \frac{1}{\varepsilon^2} , \|\Gamma_{U_{t_0}} \|^2 \right\} \left( \log \left( \frac{2}{\delta}\right) + \dimx  \right) + \bar{C}, \nonumber
\end{align}
condition~\eqref{eq:upper_bound_persis} with $ t \leftarrow t-t_0 $, $ U = U_{t_0} $, and $ \delta \leftarrow \delta/2 $ holds.
Then, it follows from Proposition~\ref{prop:LSE_performance_persis} that $ \bbP\! \left( \| \what{A}_t - A \| \le \varepsilon  \  \middle| \ \| \bar{A}_{t_0} - A \| \le \varepsilon_{t_0}  \right) \ge 1 -\delta/2 $.
Since $ \bbP ( \|  \bar{A}_{t_0} - A \| \le \varepsilon_{t_0} ) \ge  1- \delta/2  $, it holds that $ \bbP( \| \what{A}_t - A \| \le \varepsilon  ) \ge 1 -\delta $, which concludes the proof.

\section{Numerical Experiment}\label{sec:example}

We briefly examine the sample complexity of Algorithm~\ref{alg:learning} via a numerical example. In this section, the process noise is given by $ w_t \sim \calN(0,\sigma_w^2 I) $, where $ \sigma_w > 0 $.
Note that by a change of variables $ x'_t = \sigma_w^{-1} x_t $, we obtain $ x'_{t+1} = A x'_t + \sigma_w^{-1} Bu_t + w'_t $, where $ w'_t \sim \calN(0,I) $. Hence, the obtained bounds apply to this system. In the experiment, we omit resetting the state in Algorithm~\ref{alg:learning}.
Fig.~\ref{fig:sample_complexity} illustrates the relationship between the number of samples $ t $ and the estimation error $ \| A - \what{A}_t \| $ for Algorithm~\ref{alg:learning} (red), playing $ u_t \sim \calN(0,(\bar{u}/\dimu) I) $ for all $ t \ge 0 $ (yellow), and playing the oracle noise input $ u_t \sim \calN(0,U^* (A)) $ (blue), where $ U^*(A) $ is obtained by solving \eqref{prob:max_Aid_SDP_infty} with the knowledge of the true parameter $ A $. We use a Jordan block as $ A $, which is a typical example where a properly designed excitation leads to a large improvement over the isotropic noise~\cite{Wagenmaker2020}. The solid line represents the mean, while the shaded area indicates the 10th to 90th percentile range across $ 300 $ trials. 
The initial horizon is set to $ t_0 = 850 \simeq 25000^{2/3} $ such that the residual term in the upper bound \eqref{eq:upper_bound_algo} is minimized for $ t = 25000 $.
As can be seen, after the initial horizon, the error of Algorithm~\ref{alg:learning} rapidly approaches that of the oracle method, and at time $ t = 25000 $, the mean and percentiles of Algorithm~\ref{alg:learning} closely match those of the oracle method.
The numerical simulation was conducted on a MacBook Pro equipped with an Apple M3 Max chip and implemented in MATLAB R2024a, utilizing the CVX package for convex optimization~\cite{cvx}.

\begin{figure}[t]
    \centering
      \includegraphics[width=0.65\linewidth]{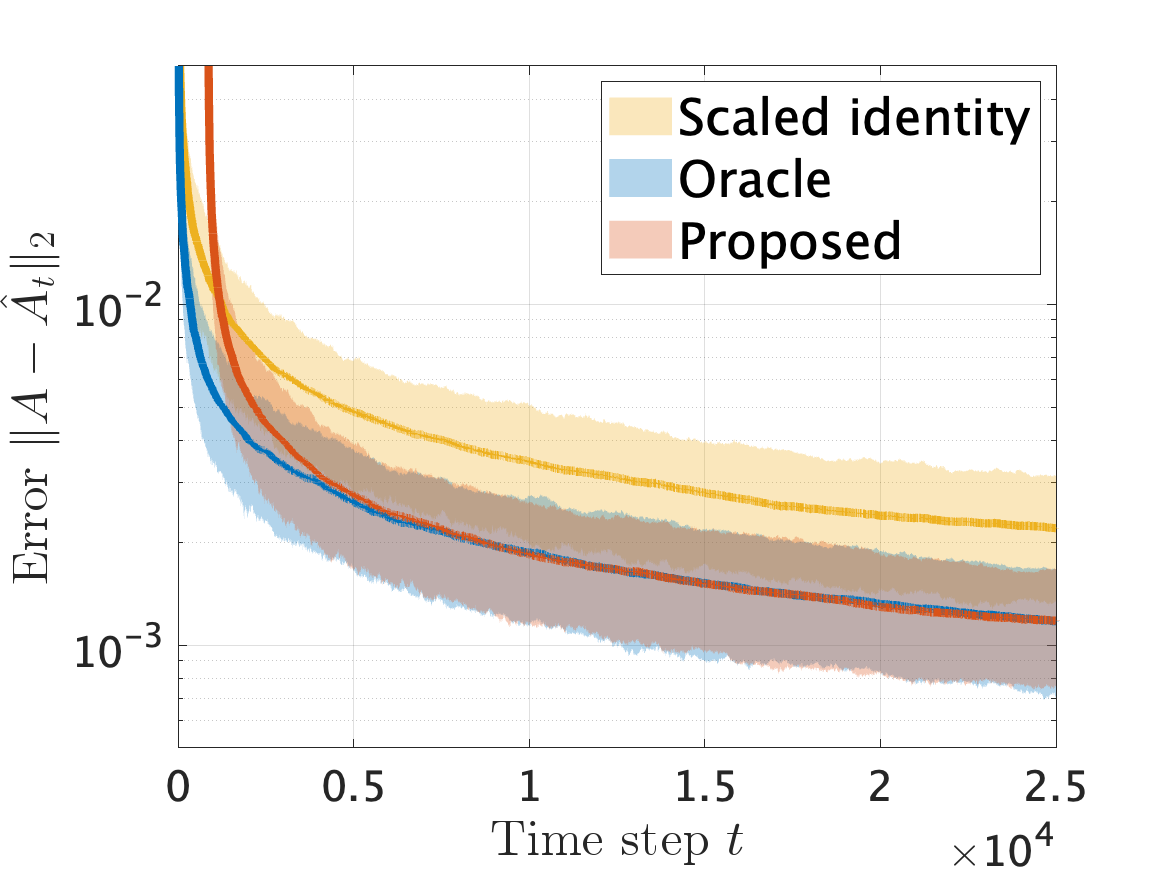}  
   \caption{Number of samples and the estimation error, where $ \dimx = 4 $, $ B = I $, $ \bar{u} = 1 $, $ \sigma_w = 0.1 $, and $ A $ is the Jordan block with diagonal entry $ 0.8 $.}
   \label{fig:sample_complexity}
\end{figure}

\section{Conclusion}\label{sec:conclusion}

In this paper, we studied the problem of active optimal excitation input design for linear system identification. We derived sample complexity lower bounds that apply to any algorithm employing zero-mean (centered) inputs, including widely used noise-based excitation strategies. Building on these insights, we proposed a computationally efficient algorithm inspired by the lower bounds and provided a matching sample complexity guarantee. This guarantee features explicit and interpretable dependence on the accuracy and confidence parameters $(\varepsilon, \delta)$ as well as key system parameters.

This work opens up several directions for future research. One promising avenue is the fixed-confidence setting, where the algorithm continues sampling until it can confidently return an estimate with a desired level of accuracy—this has been explored without excitation in \cite{Jedra2023}. Another direction is to design excitation strategies tailored to specific downstream tasks, aiming to gather only the information necessary to complete the task as efficiently as possible—a setting relevant to reinforcement learning (see \cite{Wagenmaker2021task} for example). Finally, an important remaining challenge is to remove the assumption of state resets in our upper bound analysis.


\bibliographystyle{ieeepes}
\bibliography{opt_sysid_cdc_arxiv}


\appendices

\section{Proof of Proposition~\ref{prop:lower_bound_persis}}\label{app:lower_bound}

The proof follows a similar argument as in \cite[Theorem~1]{Jedra2023}. 
Since $ w_t $ follows a non-degenerate Gaussian distribution, the conditional density function $ f_A^{x_{s+1}} ( \cdot | x_s, u_s) $ of $ x_{s+1} $ given $ x_s $ and $ u_s $ under a system matrix $ A $ exists. 
Since we suppose that the control policy satisfies \eqref{eq:density}, $ \{x_1,\ldots,x_t,u_0,\ldots,u_{t-1}\} $ under $ A $ has a joint density function
\begin{align*}
    &f_A(x_1,\ldots,x_t,u_0,\ldots,u_{t-1}) \\
    &= \prod_{s=0}^{t-1} f_A^{x_{s+1}} (x_{s+1} | x_s, u_s) \pi_s (u_s |x_1,\ldots,x_s,u_0,\ldots,u_{s-1} ) .
\end{align*}
Define the log-likelihood ratio as follows:
\begin{equation*}
   L_t := \log \! \left( \frac{f_A (x_1,\ldots,x_t,u_0,\ldots,u_{t-1})}{f_{A'} (x_1,\ldots,x_t,u_0,\ldots,u_{t-1})} \right) .
\end{equation*}
By using the Markovian nature of the dynamics, the expectation $ \bbE_A [L_t] $ under $ A $ can be written as
\begin{align}
   &\bbE_A [L_t] =  \bbE_A \! \left[ \sum_{s=1}^{t-1} D_{\rm KL} \! \left(f_A^{x_{s+1}} (\cdot | x_s, u_t) \| f_{A'}^{x_{s+1}} (\cdot | x_s, u_t) \right)  \right] \nonumber\\
   &= \bbE_A \! \left[ \sum_{s=1}^{t-1} D_{\rm KL} \! \left( \calN (Ax_s + Bu_s , I) \| \calN (A' x_s + Bu_s , I) \right)  \right] \nonumber\\
   &= \bbE_A \! \left[ \sum_{s=1}^{t-1} \frac{1}{2} x_s^\top (A - A')^\top (A - A') x_s   \right] \nonumber\\
   &= \frac{1}{2} \trace \! \left( (A - A')^\top (A - A') \sum_{s=1}^{t-1} \bbE_A \! \left[x_s x_s^\top \right] \right) , \nonumber
\end{align}
where $ D_{\rm KL}(f_1 \| f_2) $ denotes the KL divergence between distributions $ f_1 $ and $ f_2 $. Then, by the same arguments as in the proofs of \cite[Proposition~1, Theorem~1]{Jedra2023} using the data processing inequality, for any $ A'\neq A $ such that $ 2\varepsilon \le \| A - A' \| \le 6\sqrt{2} \varepsilon $, we obtain
\begin{align*}
   &\trace \! \left( (A - A')^\top (A - A') \sum_{s=1}^{t-1} \bbE_A \! \left[x_s x_s^\top \right]  \right) \\
   &\ge 2 \sup_{\calE \in \calF_t} {\rm kl} \! \left( \bbP_A(\calE), \bbP_{A'} (\calE) \right) \ge 2 \log \! \left( \frac{1}{3\delta} \right) ,\nonumber
\end{align*}
where the last inequality is obtained by considering the event $ \calE = \{ \| \what{A}_t - A \| \le \varepsilon \} \in \calF_t $.
The minimization of the left-hand side by varying $ A' $ under the constraint $ 2\varepsilon \le \| A - A' \| \le 6\sqrt{2} \varepsilon $ yields the tightest bound:
\begin{equation}
   4\varepsilon^2  \lambda_{\min} \! \left(\sum_{s=1}^{t-1} \bbE_A [x_s x_s^\top] \right) \ge 2 \log \! \left( \frac{1}{3\delta} \right) . \nonumber
\end{equation}

The evolution equation~\eqref{eq:evolution_sigma} can be obtained as follows. 
The covariance $ \Sigma_s := \bbE_A [x_s x_s^\top] $ satisfies
\begin{align}
   \Sigma_{s+1} &= A \Sigma_s A^\top + A \bbE_A [x_s u_s^\top] B^\top + B\bbE_A [u_s x_s^\top] A^\top \nonumber\\
   &\quad + B U_s B^\top + I , \label{eq:evolution2}
\end{align}
where $ U_s := \bbE_A [u_s u_s^\top] $.
By the assumption that $ \bbE[u_s | \calF_{s} ] = 0 $ and the tower property, we have
\begin{align*}
   \bbE_A [x_s u_s^\top] = \bbE_A \! \left[ \bbE_A [ x_s u_s^\top | \calF_{s} ] \right] = \bbE_A \! \left[x_s \bbE_A [u_s^\top | \calF_{s}] \right] = 0.
\end{align*}
By substituting this into \eqref{eq:evolution2}, we obtain the desired result.

\section{Proof of Theorem~\ref{thm:dim_lower_bound}}\label{app:dim_lower_bound}

For the proof of Theorem~\ref{thm:dim_lower_bound}, we utilize the following proposition~\cite[Proposition~2]{Jedra2023}.
\begin{proposition}\label{prop:bound_for_lower_bound}
    Suppose that the assumptions in Proposition~\ref{prop:lower_bound_persis} are satisfied.
    Let $ \calE_1, \calE_2,\ldots,\calE_n $ be pairwise disjoint events, which belong to $ \calF_t $. Then, for any $ A_1,A_2,\ldots,A_n \in \bbR^{\dimx \times \dimx } $ such that $ A_i \neq A $ for all $ i\in \{1,\ldots,n\} $, it holds that for any $ t \ge 1 $,
    \begin{align}
        &\frac{1}{2n} \sum_{i=1}^n \trace \! \left( (A_i - A)^\top (A_i - A) \sum_{s=1}^{t-1} \bbE_{A_i} \!\left[x_s x_s^\top \right] \right) \nonumber\\
        &\ge {\rm kl} \! \left( \frac{1}{n} \sum_{i=1}^n \bbP_{A_i} (\calE_i) , \frac{\bbP_A (\cup_{i=1}^n \calE_i)}{n}  \right) ,
    \end{align}
    where $ {\rm kl} (x,y) := x \log (x/y) + (1-x) \log ((1-x)/(1-y)) $ denotes the KL divergence between two Bernoulli distributions with mean $ x $ and $ y $, respectively.
    \hfill $ \diamondsuit $
\end{proposition}

The following sensitivity result of the Gramians under a system perturbation is also instrumental for proving Theorem~\ref{thm:dim_lower_bound}. Moreover, this result plays a crucial role in analyzing the sensitivity of the optimal solution $ U $ to the maximization of $ \lambda_{\min} ( \Gamma_\infty (A) + \Xi_\infty (A, U) ) $; see Proposition~\ref{prop:perturbation_infty_persis_modi}. 
\begin{lemma}[Perturbed Gramians]\label{lem:perturbed_gramian}
    Let $ A $ be a stable matrix. Then, for any $ \Delta \in \bbR^{\dimx \times \dimx} $ and $ \{U_k\} $ such that $ \| \Delta \| \le \| \Gamma_\infty (A) \|^{-3/2} / 4 $ and \fin{$ U_k \succeq 0 $} and $ \trace(U_k) \le \bar{u} $ for any $ k \ge 0 $, it holds that for any $ s \ge 0 $,
    \begin{align}
        &\| \Gamma_s (A + \Delta) + \Xi_s (A + \Delta,\{U_k\}) \nonumber\\
        &\quad - (\Gamma_s (A) + \Xi_s (A,\{U_k\})) \| \nonumber\\
        &\le 32 (1 + \|B\|^2 \bar{u}) \|\Delta \| \| \Gamma_\infty (A) \|^3 ,\label{eq:perturbed_gramian}
    \end{align}
    where $ \Xi_s (A,\{U_k\}) := \sum_{k=0}^{s-1} A^{s-1-k} BU_k B^\top (A^{s-1-k})^\top $.
    In addition, $ A + \Delta $ is a stable matrix.
    \hfill $ \diamondsuit $
\end{lemma}

\fin{
\begin{proof}
    Let $ P := \Gamma_\infty (A) $, $ g := \| P \| $, $ d := \| \Delta \| $, $ A_\Delta := A + \Delta $, and $ \kappa := 1 + \| B \|^2 \bar{u} $. Since $ P = APA^\top + I $, we have $ P \succeq I $ and $ g \ge 1 $. Moreover,
    \[
        \| A P^{1/2} \|^2 = \| APA^\top \| = \| P - I \| = g - 1.
    \]  
    Thus, $ E := A_\Delta P A_\Delta^\top - APA^\top $ satisfies
    \begin{align*}
        \| E \| &\le 2 \| \Delta P A^\top \| + \| \Delta P \Delta^\top \| \\
        &\le 2 d \sqrt{g(g-1)} + d^2 g \le 2dg +  d^2 g =: \theta .
    \end{align*}
    The assumption $ d \le g^{-3/2} /4  $ implies $ \theta \le 9/ 16 < 1 $. Hence,
    \[
        A_\Delta P A_\Delta^\top \preceq P - (1-\theta) I .
    \] 
    Let $ c_* := (1-\theta)^{-1} \le 16/7 $. Then, it holds that
    \[
        A_\Delta (c_* P) A_\Delta^\top + I \preceq c_* P.
    \]
    This Lyapunov inequality implies that $ A_\Delta $ is stable. Moreover, iterating the finite-horizon Lyapunov recursion from zero gives
    \begin{equation}\label{eq:Gamma_Adelta}
        \Gamma_s (A_\Delta) \preceq c_* P, \qquad s \ge 0.
    \end{equation}

    Define
    \[
        S_s (A) := \Gamma_s (A) + \Xi_s (A, \{U_k\}), \quad Q_s := I + BU_s B^\top .
    \]
    Since $ U_s \succeq 0 $ and $ \trace(U_s) \le \bar{u} $, we have
    \[
        I \preceq Q_s \preceq \kappa I .
    \]
    Consequently, by \eqref{eq:Gamma_Adelta},
    \begin{equation}\label{eq:S_P_bound}
        0 \preceq S_s (A_\Delta) \preceq \kappa \Gamma_s (A_\Delta) \preceq \kappa c_* P .
    \end{equation}

    Let $ D_s := S_s (A_\Delta) - S_s (A) $. Then, the Lyapunov equation $ S_{s+1} (A') = A' S_s (A') A'^\top + Q_s $ for $ A' = A_\Delta $ and $ A' = A $ yields
    \begin{equation}\label{eq:D_lyap}
        D_{s+1} = A D_s A^\top + R_s ,
    \end{equation}
    where
    \[
        R_s := \Delta S_s (A_\Delta) A^\top + AS_s (A_\Delta) \Delta^\top + \Delta S_s (A_\Delta) \Delta^\top .
    \]
    By \eqref{eq:S_P_bound}, there exists $ T_s $ such that $ 0 \preceq T_s \preceq \kappa c_* I $ and $ S_s (A_\Delta) = P^{1/2} T_s P^{1/2} $. Thus,
    \begin{align*}
        \| R_s \| &\le \kappa c_* \! \left( 2d \sqrt{g(g-1)} + d^2 g \right) \\
        &\le \kappa c_* dg (2+d) \le 6\kappa dg,
    \end{align*}
    where we used $ d \le 1/4 $ and $ c_* \le 16/7 $. Since $ R_s $ is symmetric, it holds that $ - r_0 I \preceq R_s \preceq r_0 I $ for $ r_0 := 6\kappa dg $. The solution $ \{D_s\} $ to \eqref{eq:D_lyap} with $ D_0 = 0 $ can be written as
    \[
        D_s = \sum_{j=0}^{s-1} A^j R_{s-1-j} (A^j)^\top ,
    \]
    and therefore, 
    \[
        - r_0 \Gamma_s (A) \preceq D_s \preceq r_0 \Gamma_s (A) .
    \]
    It follows that
    \[
        \| D_s \| \le r_0 \| \Gamma_s (A) \| \le 6\kappa d g^2 \le 32 \kappa d g^3,
    \]
    where the last inequality uses $ g \ge 1 $. This completes the proof.
\end{proof}
}

Define $ \calC_\zeta^* := \{ A + (2\varepsilon + \zeta) v u^\top : v \in S^{\dimx -1} \} $, where $ u \in S^{\dimx - 1} $ is the normalized eigenvector corresponding to the smallest eigenvalue of $ \sum_{s=0}^{t-1} \bbE_A [x_s x_s^\top] $, and $ S^{\dimx -1} $ denotes the unit sphere in $ \bbR^{\dimx} $.
Next, we construct an appropriate packing of $ \calC_\zeta^* $~\cite{Vershynin2018}.
A subset $ \calP \subseteq \calC_\zeta^* $ is called an $ \epsilon $-packing of $ \calC_\zeta^* $ if $ \| A_1 - A_2 \| > \epsilon $ for all $ A_1, A_2 \in \calP $ such that $ A_1 \neq A_2 $.
Then, there exists a packing satisfying the following property, which is a modified version of \cite[Lemma~2, Proposition~3]{Jedra2023}.
\begin{proposition}[Packing of $ \calC_\zeta^* $]\label{prop:packing}
    For $ \zeta := (3\sqrt{2} - 2) \varepsilon $ and for any $A'\in \calC_\zeta^* $, we have
    \begin{align}
        \trace\! \left( (A'- A)^\top (A'- A) G_t (A)  \right) &= 18 \varepsilon^2 \lambda_{\min} \left( G_t (A)  \right) ,\nonumber
    \end{align}
    where $ G_t (A) := \sum_{s=0}^{t-1} \bbE_{A} [x_s x_s^\top] $.
    Moreover, there exists a packing $ \calP \subseteq \calC_\zeta^* $ such that for all $ A' \in \calP $, $ \|A' - A \| = \|A' - A \|_{\rm F} = 3\sqrt{2}\varepsilon $, and for all $ A_1, A_2 \in \calP $ such that $ A_1 \neq A_2 $, we have $ 3(\sqrt{3}-1) \varepsilon \le \| A_1 - A_2 \| < 6\sqrt{2}\varepsilon $. Furthermore, it holds that
    \begin{align*}
        |\calP| \ge \left( \frac{2}{\sqrt{3}}  \right)^{\dimx (1+o(1))} .
    \end{align*}
    \hfill $ \diamondsuit $
\end{proposition}

Now, we are ready to prove Theorem~\ref{thm:dim_lower_bound}.
\paragraph*{Proof of Theorem~\ref{thm:dim_lower_bound}}
    Let $ \calP = \{A_1,\ldots,A_{|\calP|} \} \subseteq \calC_{(3\sqrt{2}-2)\varepsilon}^* $ be the packing given in Proposition~\ref{prop:packing}. For any $ i \in \{1,\ldots,|\calP| \} $, define the event $ \calE_i := \{ \|\what{A}_{\tau_A} - A_i \| \le \varepsilon  \} $.
    \fin{Since the packing points are separated by more than $ 2\varepsilon $, the events $ \calE_1,\ldots,\calE_{|\calP|} $ are pairwise disjoint. Moreover, all $ A_i $ belong to $ \calB (A, 6\sqrt{2} \varepsilon) $, and $ (\varepsilon,\delta) $-local stability implies 
    \[
        \bbP_{A_i} (\calE_i) \ge 1 -\delta .
    \]
    The event $ \calE_0 := \{ \| \what{A}_t - A  \| \le \varepsilon \} $ is disjoint from $ \cup_{i=1}^{|\calP|} \calE_i $, because $ \| A_i - A\| = 3\sqrt{2} \varepsilon > 2\varepsilon $. Hence, the local stability also gives 
    \[
        \bbP_A  \Biggl( \bigcup_{i=1}^{|\calP|} \calE_i \Biggr) \le P_A (\calE_0^c) \le \delta .
    \]
    }

    \fin{
    Consequently,
    \[
        p:= \frac{1}{|\calP|} \sum_{i=1}^{|\calP|} \bbP_{A_i} (\calE_i), \qquad q:= \frac{\bbP_A(\cup_{i=1}^{|\calP|} \calE_i)}{|\calP|} 
    \]
    satisfy $ p \ge 1-\delta > 1/2 $ and $ q \le \delta / |\calP| < 1/4 $. If $ q = 0 $, the KL divergence $ {\rm kl} (p,q) $ is infinite and the following bounds are immediate. Otherwise, using $ h(p) \le \log 2 $ for the binary entropy function $ h(p) := -p \log p - (1-p) \log (1-p) $, we have
    \begin{align}
        {\rm kl} (p,q) &= p \log \frac{1}{q} + (1-p) \log \frac{1}{1-q} - h(p) \nonumber\\
        &\ge \frac{1}{2} \log \! \left( \frac{|\calP|}{4\delta} \right) .\label{eq:kl_lower1}
    \end{align}
    In addition, by Pinsker's inequality, it holds that
    \begin{equation*}
        {\rm kl} (p,q) \ge 2 (p-q)^2 \ge \frac{1}{8} .
    \end{equation*}
    Combining this with \eqref{eq:kl_lower1}, we obtain
    \[
        {\rm kl} (p,q) \ge \frac{1}{16} \log \! \left( \frac{|\calP|}{\delta} \right).
    \]
    Thus, it follows from Proposition~\ref{prop:bound_for_lower_bound} that
    \begin{equation}
        \frac{1}{2 |\calP|} \sum_{i=1}^{|\calP|} \trace \left( (A_i - A)^\top (A_i - A) G_t(A_i) \right) \ge \frac{1}{16} \log \! \left( \frac{|\calP|}{\delta} \right) . \label{eq:bound_dim1}
    \end{equation}
    }

    \fin{
    We next compare $ G_t (A_i) $ and $ G_t (A) $. For $ A' \in \calB (A, 6\sqrt{2} \varepsilon) $, define
    \begin{align*}
        U_t^{A'} &:= \bbE_{A'} [u_t u_t^\top], \\
        H_t (A') &:= \sum_{s=1}^{t-1} \Gamma_s (A'), \qquad \kappa := 1 + \| B\|^2 \bar{u} .
    \end{align*}
    By Assumption~\ref{ass:power}, it holds that
    \[
        0 \preceq U_s^{A'} \preceq \trace(U_s^{A'}) I \preceq \bar{u} I .
    \]
    Consequently, for any $ t \ge 1 $, 
    \begin{align*}
        \Gamma_t (A') &\preceq \bbE_{A'} [x_t x_t^\top] \\
        &= \sum_{s=0}^{t-1} (A')^{t-1-s} (I + BU_s^{A'} B^\top)(A'^\top)^{t-1-s} \\
        &\preceq \kappa \Gamma_t (A') ,
    \end{align*}
    and hence,
    \begin{equation}\label{eq:H_G}
        H_t (A') \preceq G_t (A') \preceq \kappa H_t (A') .
    \end{equation}

    Let $ g := \| \Gamma_\infty (A) \| $ and $ r := 3\sqrt{2} \varepsilon $. It follows from $ g \ge 1 $ and the assumption $ \varepsilon \le g^{-3} /(12\sqrt{2}) $ that $ r < g^{-3} / 4 \le g^{-3/2} / 4 $. Then, by Proposition~\ref{prop:packing}, the perturbation $ \Delta = A_i - A $ satisfies $ \| \Delta \| = r  \le \| \Gamma_\infty (A) \|^{-3/2} / 4 $, and applying Lemma~\ref{lem:perturbed_gramian} to the passive case $ U_k \equiv 0 $ with $ \Delta = A_i - A $ and $ \bar{u} = 0 $ yields
    \[
        \| \Gamma_t (A_i) - \Gamma_t (A) \| \le 32 r g^{3} \le 8, \qquad \forall t \ge 1.
    \]
    Since $ \Gamma_t (A) \succeq I $, it follows that
    \[
        \Gamma_t (A_i) \preceq \Gamma_t (A) + 8I \preceq 9\Gamma_t (A) ,
    \]
    and therefore, $ H_t (A_i) \preceq 9 H_t (A) $. Combining this with \eqref{eq:H_G}, we have
    \begin{equation}\label{eq:G_kappa}
        G_t (A_i) \preceq 9\kappa G_t (A) .
    \end{equation}

    By the construction of $ \calC_{(3\sqrt{2} - 2)\varepsilon}^* $, we can write $ A_i - A = r v_i u^\top $, where $ v_i \in S^{n_x -1} $ and $ u \in S^{n_x -1} $ is the eigenvector associated with $ \lambda_{\min} (G_t (A)) $.
    Thus, by \eqref{eq:G_kappa},
    \begin{align*}
        \trace \! \left( (A_i - A)^\top (A_i - A) G_t (A_i) \right) &= r^2 u^\top G_t (A_i) u  \\
        &\le 9 \kappa r^2 u^\top G_t (A) u \\
        &= 162 \kappa \varepsilon^2 \lambda_{\min} (G_t (A)) .
    \end{align*} 
    Combining this inequality with \eqref{eq:bound_dim1}, we obtain
    \begin{align*}
        \lambda_{\min} (G_t (A)) &\ge \frac{1}{1296\kappa \varepsilon^2} \log \! \left( \frac{|\calP|}{\delta} \right) \\
        &\ge  \frac{1}{1296\kappa \varepsilon^2} \\
        &\hspace{-0.2cm}\times \left( \log \! \left( \frac{1}{\delta} \right) + \dimx (1+o(1))  \log \! \left( \frac{2}{\sqrt{3}}  \right) \right) ,
    \end{align*}
    which completes the proof.
    }

\section{Proof of Corollaries~\ref{Cor:lower_bound_infty} and \ref{Cor:lower_bound_infty_dim}}\label{app:lower_bound_infty}
Under any centered control policy $ \{ \pi_s \} $ with $ U_s := \bbE[u_s u_s^\top] $, it holds that
\begin{align*}
\sum_{s=1}^{\tau-1} \Sigma_s &= \sum_{s=0}^{\tau-2} \biggl( (\tau -1 -s) A^s (A^s)^\top \\
&\quad + \sum_{k=0}^{\tau - 2-s} A^{k} B U_s B^\top (A^{k})^\top \biggr).
\end{align*} 
Hence we have: for any $ \{U_s\} $,
\begin{align*}
\sum_{s=1}^{\tau-1} \Sigma_s &\preceq \sum_{s=0}^{\tau-2} \left( (\tau -1 -s) A^s (A^s)^\top + \Xi_\infty (A,U_s) \right) \\
&\preceq (\tau -1) \Gamma_\infty(A) + \sum_{s=0}^{\tau-2} \Xi_\infty (A,U_s) .
\end{align*}
We deduce that: 
\begin{align}
   &\max_{\trace(U_s) \le \bar{u}, \forall s} \lambda_{\min} \! \left(\sum_{s=1}^{\tau-1} \Sigma_s \right) \nonumber\\
   &\le \max_{\trace(U_s) \le \bar{u},  \forall s} \lambda_{\min} \! \left((\tau -1) \Gamma_\infty(A) + \sum_{s=0}^{\tau-2} \Xi_\infty (A,U_s) \right) \nonumber\\
   &= (\tau -1 ) \max_{\trace(U) \le \bar{u}} \lambda_{\min} \! \left(\Gamma_\infty(A) +  \Xi_\infty (A,U) \right) .
\end{align}
Combining this result with Proposition~\ref{prop:lower_bound_persis} and Theorem~\ref{thm:dim_lower_bound}, we obtain the desired results.

\section{Proof of Proposition~\ref{prop:LSE_performance_persis}}\label{app:LSE_performance}
Define $ Q_U := BUB^\top + I $ and
\begin{align}
    \Gamma_{U} &:= 
    \begin{bmatrix}
        Q_U^{1/2} & 0 & \cdots & 0 \\
        AQ_U^{1/2} & Q_U^{1/2} & \ddots & 0 \\
        \vdots & AQ_U^{1/2} & \ddots & 0 \\
        A^{t-2}Q_U^{1/2} & \cdots & AQ_U^{1/2} & Q_U^{1/2} 
    \end{bmatrix}  \nonumber\\
    M_{U} &:= \! \left( \sum_{s=1}^{t-1} \Gamma_s(A) + \sum_{s=1}^{t-1} \Xi_s (A,U)  \right)^{-1/2} , \nonumber\\
    \xi &:= \left[ \eta_0^\top , \eta_1^\top,\ldots, \eta_{t-2}^\top \right]^\top , \nonumber
\end{align}
where $ \eta_t := Q_U^{-1/2} (Bu_t + w_t) $ for $ t\ge 0 $.
As can be seen from \eqref{eq:error_closed_persis}, $ X^\top X $ is closely related to the performance of the LSE.
The following concentration inequality is useful for deriving a sample complexity upper bound of the LSE in terms of $ M_{U} $.

\begin{lemma}\label{lem:isometry_persis}
    Suppose that the assumptions in Proposition~\ref{prop:LSE_performance_persis} hold.
    For $ X := [x_1,x_2,\ldots,x_{t-1}]^\top $ and $ M_{U} $, it holds that for any $ \varepsilon > 0 $,
    \begin{equation*}
        \left\| (XM_{U})^\top XM_{U}  - I \right\| > \max\{\varepsilon, \varepsilon^2 \} K^2
    \end{equation*}
    with probability at most
    \begin{equation*}
        2\exp\! \left( - c_1 \varepsilon^2 \frac{1}{\|M_{U}\|^2 \|\Gamma_{U} \|^2} + c_2 \dimx \right)
    \end{equation*}
    for some positive universal constants $ c_1, c_2 $.
    \hfill $ \diamondsuit $
\end{lemma}
\begin{proof}
    Noting that $ {\rm vec}(X^\top) = \Gamma_{U} \xi  $ and $ \|\sigma_{Mv}^\top \Gamma_{U} \|_{\rm F}^2 = \trace(\Gamma_{U}^\top \sigma_{Mv} \sigma_{Mv}^\top \Gamma_{U}) = 1  $ for any $ v \in S^{\dimx -1} $, where $ \sigma_{Mv} := \diag (M_{U} v,M_{U}v,\ldots,M_{U}v) \in \bbR^{(t-1)\dimx \times (t-1)} $, by the same argument as in \cite[Lemma~2]{Jedra2020}, we obtain the desired result.
\end{proof}

Proposition~\ref{prop:LSE_performance_persis} follows from the same argument as in the proof of \cite[Theorem~1]{Jedra2020}, provided that \cite[Lemma~1]{Jedra2020} is replaced by Lemma~\ref{lem:isometry_persis}.

\section{Proof of Proposition~\ref{prop:perturbation_infty_persis_modi}}\label{app:perturbation_modi}
By Lemma~\ref{lem:perturbed_gramian} and the assumption that $ A $ is stable and $ \| \Delta \| \le \| \Gamma_\infty(A) \|^{-3/2} /4 $, for any $ U \succeq 0 $ such that $ \trace(U) \le \bar{u} $, we have
    \begin{align}
        &\left\| \Gamma_\infty (A) + \Xi_\infty (A,U) - \left(\Gamma_\infty (A + \Delta) + \Xi_\infty (A + \Delta,U)\right) \right\| \nonumber\\
        &\le 32 (1 + \| B\|^2 \bar{u}) \| \Delta \| \| \Gamma_\infty(A) \|^3 . \nonumber
    \end{align}
Therefore, for any $ U \succeq 0 $ such that $ \trace(U) \le \bar{u} $, we obtain
    \begin{align}
        &J_{A} (U) - J_{A + \Delta} (U) \nonumber\\
        &= \lambda_{\min} \! \left( \Gamma_\infty (A) + \Xi_\infty (A,U) \right) \nonumber\\
        &\quad - \lambda_{\min} \! \left( \Gamma_\infty (A+\Delta) + \Xi_\infty (A+\Delta,U) \right) \nonumber\\
        &\le \| \Gamma_\infty (A) + \Xi_\infty (A,U) \nonumber\\
        &\quad - \left(\Gamma_\infty (A+\Delta) + \Xi_\infty (A+\Delta,U)\right)  \|\nonumber\\
        &\le 32 (1 + \| B\|^2 \bar{u}) \| \Delta \| \| \Gamma_\infty(A) \|^3, \nonumber
    \end{align}
    where we used Weyl's inequality for symmetric matrices.
    Lastly, by the optimality of $ U^*(A+\Delta) $ for $ J_{A+\Delta} $, it holds that $ J_{A+\Delta} (U^*(A)) - J_{A+\Delta} (U^*(A+\Delta)) \le 0 $ and
    \begin{align}
        &J_{A} (U^* (A)) -  J_{A} (U^* (A+\Delta)) \nonumber\\
        &= [J_{A} (U^* (A)) - J_{A+\Delta} (U^*(A))] \nonumber\\
        &\quad + [J_{A+\Delta} (U^*(A)) - J_{A+\Delta} (U^*(A+\Delta))] \nonumber\\
        &\quad +  [J_{A+\Delta} (U^*(A+\Delta)) -  J_{A} (U^* (A+\Delta)) ] \nonumber\\
        &\le [J_{A} (U^* (A)) - J_{A+\Delta} (U^*(A))] \nonumber\\
        &\quad +  [J_{A+\Delta} (U^*(A+\Delta)) -  J_{A} (U^* (A+\Delta)) ] \nonumber\\
        &\le 64 (1 + \| B\|^2 \bar{u}) \| \Delta \| \| \Gamma_\infty(A) \|^3 , \nonumber
    \end{align}
which completes the proof.

\section{Proof of Theorem~\ref{thm:upper_bound}}\label{app:upper_bound}

We bound the difference between $ (t- 1) \lambda_{\min} (\Gamma_\infty (A) + \Xi_\infty (A,U^* (A))) $ and $ \lambda_{\min} \! \left(\sum_{s=1}^{t-t_0 -1} ( \Gamma_s(A) + \Xi_s (A,U^*(\bar{A}_{t_0})) ) \right) $ step by step and employ Proposition~\ref{prop:LSE_performance_persis} to obtain the sample complexity upper bound~\eqref{eq:upper_bound_algo}.
Let $ \varepsilon_{t_0} > 0 $ be the minimum $ \varepsilon $ such that \eqref{eq:upper_bound_persis} with $ t = t_0 $, $ U = (\bar{u}/\dimu) I $, and $ \delta \leftarrow \delta /2 $ holds.
Then, by Proposition~\ref{prop:LSE_performance_persis} and Lemma~\ref{lem:perturbed_gramian}, $ \| \what{A}_{t_0} - A \| \le \varepsilon_{t_0} $ holds and $ \what{A}_{t_0} $ is stable (i.e., $ \bar{A}_{t_0} = \what{A}_{t_0} $) with probability at least $ 1 - \delta/2 $.
Since by assumption, $ \Delta := \bar{A}_{t_0} - A $ satisfies $ \bbP(\| \Delta \|  \le \| \Gamma_\infty(A) \|^{-3/2} /4 ) \ge 1 - \delta/2 $, by Proposition~\ref{prop:perturbation_infty_persis_modi}, with probability at least $ 1- \delta/2 $, $ U_{t_0} := U^* (\bar{A}_{t_0}) $ satisfies
\begin{align}
   &(t-1) \lambda_{\min} (\Gamma_\infty (A) + \Xi_\infty (A,U^* (A))) \nonumber\\
   &\quad - (t-1) \lambda_{\min} (\Gamma_\infty (A) + \Xi_\infty (A,U_{t_0})) \nonumber\\
   &\le  64 (1 + \| B\|^2 \bar{u})  \| \Gamma_\infty(A) \|^3 (t-1) \varepsilon_{t_0}. \label{eq:bound_step1_modi}
\end{align}

In addition, the difference between the minimum eigenvalues of the infinite and finite time Gramians is bounded as follows:
\begin{align}
   &(t-1) \lambda_{\min} (\Gamma_\infty (A) + \Xi_\infty (A,U_{t_0})) \nonumber\\
   &\quad - (t-1) \lambda_{\min} \! \left( \sum_{s=0}^{t-2} A^s (A^s)^\top + \sum_{s=0}^{t-2} A^s B U_{t_0} B^\top (A^s)^\top  \right) \nonumber\\
   &\le (t-1) \left\| \sum_{s=t-1}^\infty A^s (A^s)^\top + \sum_{s=t-1}^\infty A^s B U_{t_0} B^\top (A^s)^\top   \right\| \nonumber\\
   &\le \frac{(1+\|B\|^2 \bar{u}) (t-1)\gamma^{t-1}}{(1-\gamma)^2}, \label{eq:bound_step2_modi}
\end{align}
where we used the fact that for $ \gamma :=  1 - \| \Gamma_\infty (A) \|^{-1} \in [0,1) $ and for any $ k \ge 1 $,
\begin{equation*}
   \| A^k \| \le \frac{\gamma^{k/2}}{(1-\gamma)^{1/2}}.
\end{equation*}

Moreover, noting that $ \sum_{s=1}^{t-1} (\Gamma_s(A) + \Xi_s(A,U_{t_0}) ) = \sum_{s=0}^{t-2}  (t-1-s) (A^s (A^s)^\top + A^s B U_{t_0} B^\top (A^s)^\top ) $, we have
\begin{align}
   &(t-1) \lambda_{\min} \! \left( \sum_{s=0}^{t-2} A^s (A^s)^\top + \sum_{s=0}^{t-2} A^s B U_{t_0} B^\top (A^s)^\top  \right)  \nonumber\\
   &\quad - \lambda_{\min} \! \left( \sum_{s=1}^{t-1} (\Gamma_s(A) + \Xi_s(A,U_{t_0}) ) \right) \nonumber\\
   &\le \left\| \sum_{s=0}^{t-2} s \! \left(A^s (A^s)^\top + A^s B U_{t_0} B^\top (A^s)^\top \right)  \right\| \nonumber\\
   &\le \frac{ (1 + \|B\|^2 \bar{u}) \gamma }{(1- \gamma)^3} \nonumber\\
   &= (1 + \|B\|^2 \bar{u}) \|\Gamma_\infty(A) \|^3 (1- \|\Gamma_\infty(A) \|^{-1}). \label{eq:bound_step3_modi}
\end{align}

Furthermore, we have
\begin{align}
   & \lambda_{\min} \! \left( \sum_{s=1}^{t-1} (\Gamma_s(A) + \Xi_s(A,U_{t_0}) ) \right)  \nonumber\\
   &\quad - \lambda_{\min} \! \left( \sum_{s=1}^{t-t_0 - 1} (\Gamma_s(A) + \Xi_s(A,U_{t_0}) ) \right) \nonumber\\
   & \le \left\|  \sum_{s=t - t_0}^{t-1} (\Gamma_s (A) + \Xi_s (A,U_{t_0}) ) \right\| \nonumber \\
   & \le (1 + \| B\|^2 \bar{u}) \sum_{s=t-t_0}^{t-1} \sum_{k=0}^{s-1} \| A^k \|^2 \nonumber\\
   &\le \frac{(1 + \| B\|^2 \bar{u})t_0}{(1-\gamma)^2} = (1 + \|B\|^2 \bar{u}) \|\Gamma_\infty (A) \|^2 t_0 . \label{eq:bound_step4_modi}
\end{align}

By \eqref{eq:bound_step1_modi}--\eqref{eq:bound_step4_modi}, we arrive at
\begin{align}
   &(t-1) \lambda_{\min} (\Gamma_\infty (A) + \Xi_\infty (A,U^* (A))) \nonumber\\
   &\quad - \lambda_{\min} \! \left( \sum_{s=1}^{t-t_0 - 1} ( \Gamma_s(A) + \Xi_s(A,U_{t_0}) ) \right) \nonumber\\
   &\le \biggl( C_1  (t-1)\varepsilon_{t_0} +  C_2  (t-1) (1-\|\Gamma_\infty(A) \|^{-1})^{t-1} \nonumber\\
   &\quad + C_2 (\|\Gamma_\infty(A) \| - 1)  + C_2 t_0  \biggr) \| \Gamma_\infty(A) \|^2 , \nonumber
\end{align}
where $ C_1 := 64 (1 + \| B\|^2 \bar{u})  \| \Gamma_\infty(A) \| $ and $ C_2 := 1+\|B\|^2 \bar{u}  $.

Therefore, if $ \| \bar{A}_{t_0} - A \| \le \varepsilon_{t_0} $ and
\begin{align}
   &(t-1) \lambda_{\min} (\Gamma_\infty (A) + \Xi_\infty (A,U^* (A))) \nonumber\\
   &\ge C \max \left\{ \frac{1}{\varepsilon^2} , \|\Gamma_{U_{t_0}} \|^2 \right\} \left( \log \left( \frac{2}{\delta}\right) + \dimx  \right) \nonumber\\
   &\quad +  \Bigl( C_1  (t-1)\varepsilon_{t_0} +  C_2 (t-1) (1-\|\Gamma_\infty(A) \|^{-1})^{t-1} \nonumber\\
   &\quad + C_2 (\|\Gamma_\infty(A)\| - 1) + C_2 t_0 \Bigr)\| \Gamma_\infty(A) \|^2, \nonumber
\end{align}
then it holds that
\begin{align}
   &\lambda_{\min} \! \left( \sum_{s=1}^{t-t_0 - 1} \left( \Gamma_s(A) + \Xi_s (A, U_{t_0}) \right) \right)  \nonumber\\
   &\ge C \max \! \left\{ \frac{1}{\varepsilon^2} , \|\Gamma_{U_{t_0}} \|^2 \right\} \left( \log \! \left( \frac{2}{\delta}\right) + \dimx  \right), \nonumber
\end{align}
and it follows from Proposition~\ref{prop:LSE_performance_persis} that $ \bbP\! \left( \| \what{A}_t - A \| \le \varepsilon  \  \middle| \ \| \bar{A}_{t_0} - A \| \le \varepsilon_{t_0}  \right) \ge 1 -\delta/2 $.
Since $ \bbP ( \|  \bar{A}_{t_0} - A \| \le \varepsilon_{t_0} ) \ge  1- \delta/2  $, it holds that $ \bbP( \| \what{A}_t - A \| \le \varepsilon  ) \ge 1 -\delta $.

For $ t \in \{0,\ldots,t_0-1\} $, we have $ \bbE[\|u_t\|^2] = \trace ( (\bar{u}/\dimu) I ) = \bar{u} $, and for $ t \ge t_0 $, it holds that
\begin{equation*}
    \bbE[\|u_t\|^2] = \trace (\bbE[U_{t_0}]) \le \bar{u} ,
\end{equation*}
which concludes the proof.


\end{document}